\documentclass[12pt]{amsart}
\usepackage{a4wide}
\usepackage{amsthm,amsfonts,amsmath,mathrsfs,amssymb}

\newtheorem{theorem}{Theorem}
\newtheorem{lemma}{Lemma}

\theoremstyle{definition}

\newtheorem{question}{Question}

\newcommand{\Hi}{{\mathscr{H}}^\infty}
\newcommand{\Ht}{{\mathscr{H}}^2}
\newcommand{\Hp}{{\mathscr{H}}^p}
\newcommand{\D}{\mathbb{D}}
\newcommand{\C}{\mathbb{C}}
\newcommand{\T}{\mathbb{T}}
\newcommand{\real}{{\mathbf R}}
\newcommand{\complex}{{\mathbf C}}

\newcommand{\nanu}{{\mathbf N}}

\DeclareMathOperator{\supp}{{\rm supp\,}}
\newcommand{\epsi}{\varepsilon}

\newcommand{\beqla}[1] {\begin {eqnarray}\label{#1}}
\def \eeq {\end {eqnarray}}
\newcommand{\refeq}[1]{(\ref{#1})}
\newcommand{\beqno}{\begin{eqnarray*}}
\newcommand{\eeqno}{\end{eqnarray*}}

\begin{document}

\title[Integral means and boundary limits of Dirichlet series]
{Integral means and boundary limits\\ of Dirichlet series}

\author{Eero Saksman}
\address{Department of Mathematics and Statistics, University of Helsinki,
P.O. Box 68 (Gustaf H\"allstr\"omin katu 2b), FI-00014 University of
Helsinki, Finland} \email{eero.saksman@helsinki.fi}

\author{Kristian Seip}
\address{Department of Mathematical Sciences, Norwegian University of Science and Technology,
NO-7491 Trondheim, Norway} \email{seip@math.ntnu.no}
\thanks{The first author is supported by the Academy of Finland, projects
no.\ 113826 and 118765. The second author is
supported by the Research Council of Norway grant 160192/V30.
 This research is part of the European Science
Foundation Networking Programme ``Harmonic and Complex Analysis and
Applications HCAA}

\date{November 14, 2007.}


\begin{abstract}
We study the boundary behavior of functions in the Hardy spaces
$\Hp$ for ordinary Dirichlet series. Our main result, answering a
question of H. Hedenmalm, shows that the classical F. Carlson
theorem on integral means does not extend to the imaginary axis for
functions in $\Hi$, i.e., for ordinary Dirichlet series in
$H^\infty$ of the right half-plane. We discuss an important
embedding problem for $\Hp$, the solution of which is only known
when $p$ is an even integer. Viewing $\Hp$ as Hardy spaces of the
infinite-dimensional polydisc, we also present analogues of Fatou's
theorem.
\end{abstract}

\maketitle

\section{Introduction}

A classical theorem of F. Carlson \cite{C22} says that if an
ordinary Dirichlet series \begin{equation} \label{dirichlet}
f(s)=\sum_{n=1}^\infty a_n n^{-s} \end{equation} converges in the
right half-plane $\Re s>0$ and is bounded in every half-plane $\Re
s\ge \delta>0$, then for each $\sigma>0$,
\begin{equation} \label{means}  \lim_{T\to \infty}\frac{1}{T}\int_0^T |f(\sigma+i t)|^2
dt= \sum_{n=1}^\infty |a_n|^2 n^{-2\sigma}.
 \end{equation} From a
modern viewpoint, Carlson's theorem is a special case of the
general ergodic theorem, as will be explained below.

A natural question, first raised by H. Hedenmalm \cite{H04}, is
whether the identity \eqref{means} remains valid when $\sigma=0$,
provided $f(s)$ is a bounded function in $\Re s>0$. The problem
makes sense because we may replace $f(\sigma+it)$ by the
nontangential limit $f(it)$, which in this case exists for almost
every $t$. We note that the general ergodic theorem is of no help
for this problem.

We denote by $\Hi$ the class of functions $f(s)$ that are bounded
in $\Re s>0$ with $f$ represented by an ordinary Dirichlet series
\eqref{dirichlet} in some half-plane. We will use the notation
\[ \| f\|_\infty =\sup_{\sigma >0} |f(\sigma+it)| \ \ \ \text{and}
\ \ \ \|f\|_{2}^2= \sum_{n=1}^\infty |a_n|^2.\] Our main result is
that there is no ``boundary version'' of Carlson's theorem:

\begin{theorem}\label{th1} The following two statements hold:
\begin{itemize}
\item[(i)] There exists a function $f$ in $\Hi$ such that
\[ \lim_{T \to \infty} \frac{1}{T} \int_0^T |f(it)|^2 dt \]
does not exist. \item[(ii)] Given $\epsi >0$, there exists a
singular inner function $g$ in $\Hi$ such that $\| g\|_{2}\leq \epsi
.$
\end{itemize}
\end{theorem}

In other words, the limit on the left-hand side of \eqref{means} may
fail to exist, and even it does exist, the identity \eqref{means}
need not hold. In the next section, we will see that both parts of
the theorem rely on a basic construction of W. Rudin \cite{R69}
concerning radial limits of analytic functions in polydiscs.


To see how to obtain Carlson's theorem as a special case of the
general ergodic theorem, we resort to a fundamental observation of
Bohr \cite{B13b}. We put
\[ z_1=2^{-s},\ z_2=3^{-s}, \, ...,\, z_j=p_j^{-s}, \, ..., \]
where $p_j$ denotes the $j$-th prime; then, in view of the
fundamental theorem of arithmetic, the Dirichlet series
\eqref{dirichlet} can be considered as a power series in infinitely
many variables. For a given Dirichlet series $f$ we denote by $F$
the corresponding extension to the infinite polydisc $\D^\infty$;
then if $F$ happens to be a function of only $n$ variables, it is
immediate from Kronecker's theorem and the maximum principle that
\begin{equation} \|f\|_\infty = \|F\|_\infty , \label{assoc}
\end{equation} where the norm on the right-hand side is the
$H^\infty(\D^n)$ norm. The result is the same in the
infinite-dimensional case, but some care has to be taken when
defining the norm in the polydisc. (See \cite{HLS97} for details.)
We can now think of any vertical line $t\mapsto \sigma+it$ as an
ergodic flow on the infinite-dimensional torus $\T^\infty$:
\[
(\tau_1,\tau_2,\ldots ) \mapsto (p_1^{-i t}\tau_1, p_2^{-i t}\tau_2,
...) \quad \mbox{for}\;\; (\tau_1,\tau_2,\ldots )\in \T^\infty .
\] If $F(p_1^{-\sigma}z_1, p_2^{-\sigma} z_2, ...)$ is continuous on
$\T^\infty$, then the general ergodic theorem yields \eqref{means}.

A similar problem concerning integral means of nontangential limits
can be stated for the closely related space $\Ht$, which consists of
those Dirichlet series of the form \eqref{dirichlet} for which
$\|f\|_2<\infty$. In this case, $f(s)/s$ belongs to the Hardy space
$H^2$ of the half-plane $\sigma>1/2$, thanks to the following
embedding (see \cite[p. 140]{M94}, \cite[Theorem 4.1]{HLS97}):
\begin{equation} \label{embedding} \int_{\theta}^{\theta+1}\left|f\left(\frac{1}{2}+i
t\right)\right|^2 dt \le C \|f\|_2^2,
\end{equation}
with $C$ an absolute constant independent of $\theta$. It follows
immediately that we have \begin{equation}\label{pmeans} \lim_{T\to
\infty}\frac{1}{T}\int_0^T \left|f\left(\frac{1}{2}+i
t\right)\right|^2 dt= \sum_{n=1}^\infty |a_n|^2 n^{-1}\end{equation}
for every function $f$ in $\Ht$, since the space of Dirichlet
polynomials is dense in $\Ht$ and the identity holds trivially when
$f$ is a Dirichlet polynomial.

It is interesting to compare our Theorem~1 with what has been proved
about pointwise convergence of Dirichlet series in $\Ht$ and in
$\Hi$. Hedenmalm and Saksman \cite{HS03} showed that the Dirichlet
series of a function in $\Ht$ converges almost everywhere on the
vertical line $\sigma=1/2$. (See \cite{KQ01} for a short proof that
gives the result as a corollary of L. Carleson's theorem on almost
everywhere convergence of Fourier integrals.) On the other hand, F.
Bayart, S. V. Konyagin, and H. Queff\'{e}lec \cite{BKQ} exhibited an
example of a function $f$ in $\Hi$, continuous in the closed
half-plane $\sigma\ge 0$, whose Dirichlet series diverges everywhere
on the imaginary axis $\sigma=0$. Our result is consistent with
these findings: ``less'' remains of the Dirichlet series on the
boundary in the $\Hi$ setting than in the $\Ht$ setting.

In Section~3 of this paper, after the proof of Theorem~1, we will
discuss the curious situation that occurs when we replace $\Ht$ by
the spaces $\Hp$ ($1\le p < \infty$), which were introduced and
studied by Bayart \cite{B00}. The question of whether there is a
$p$-analogue of \eqref{embedding} for every $p\ge 1$ appears as the
most important problem regarding $\Hp$. This problem seems to
require quite nontrivial analytic number theory. At present, beyond
$p=2$, the result is known only when $p$ is an even integer, which
is just a trivial extension of \eqref{embedding}. Our discussion of
this problem will draw attention to those properties of $\Hp$ that
force us to abandon the standard analytical approach involving
interpolation techniques.

Finally, in Section~4, we will present certain analogues of the
Fatou theorem for Hardy spaces of the infinite-dimensional polydisc.
Our version of the Fatou theorem for $\Hp$ gives sense to the
statement that the $p$-analogue of \eqref{embedding} holds if and
only if the $p$-analogue of \eqref{pmeans} holds.

\section{Proof of Theorem \ref{th1}}

Before embarking on the proof of Theorem~1, we make some simple
observations in order to clarify what our problem is really about.
We note that another way of phrasing Hedenmalm's question is to ask
whether we have
\[   \lim_{T\to \infty}\frac{1}{T}\int_0^T |f(i t)|^2
dt=\lim_{\sigma\to 0^+} \lim_{T\to \infty}\frac{1}{T}\int_0^T
|f(\sigma+i t)|^2 dt \] for every $f$ in $\Hi$. We observe that
for a finite interval, say for $t_1<t_2$, we have indeed
\[
\int_{t_1}^{t_2} |f(i t)|^2 dt = \lim_{\sigma\to 0^+}
\int_{t_1}^{t_2} |f(\sigma+i t)|^2 dt, \] as follows by Lebesgue's
dominated convergence theorem. Similarly, by applying Cauchy's
integral theorem and again Lebesgue's dominated convergence
theorem, we get
\[ a_n=\lim_{T\to \infty} \frac{1}{T}\int_0^T f(i t)n^{it}
dt, \] for every positive integer $n$. Let us also note that the
upper estimate
\[ \liminf_{T\to \infty}\frac{1}{T}\int_0^T |f(i t)|^2
dt\ge \lim_{\sigma\to 0^+} \lim_{T\to \infty}\frac{1}{T}\int_0^T
|f(\sigma+i t)|^2 dt =\|f\|_2^2
 \] may be obtained from the Poisson integral
representation of $[f(\sigma+i t)]^2$, i.e.,
\[ [f(\sigma+i t)]^2=\frac{1}{\pi} \int_{-\infty}^\infty
[f(i\tau)]^2 \frac{\sigma}{(t-\tau)^2+\sigma^2} dt. \] We conclude
from these observations that the counterexamples of Theorem~1 should
be functions whose nontangential limits have increasing oscillations
when the argument $t$ tends to $\infty$.

We begin by recalling some terminology and briefly reviewing Rudin's
method for constructing real parts of analytic functions in the
polydisc $\D^n$ with given boundary values almost everywhere on the
distinguished boundary $\T^n$. Rudin treats $\D^n$ with arbirary
$n\geq 1$, but we shall need only the case $n=2$. We refer to
\cite[pp. 34--36]{R69} for full details of the construction.

We employ the complex notation for points on the distinguished
boundary $\T^2$ of the bidisc $\D^2.$  The normalized Lebesgue
measure on $\T^2$ is denoted by $m_2$. The distance between
$\tau=(\tau_1 ,\tau_2)$ and $\tau'=(\tau'_1 ,\tau'_2)$ is
$$d
(\tau,\tau'):= \max (|\tau_1-\tau'_1|,|\tau_2-\tau'_2|),
$$
 and $B(\tau ,r)$ stands for the  ball with center $\tau$ and radius $r$.
We set
$$
P_r(\tau):=\frac{(1-r^2)^2}{|1-r\tau_1|^2|1-r\tau_2|^2},\quad
0<r<1,
$$
where $\tau=(\tau_1 ,\tau_2)$ is a point in $\T^2$. In particular,
the Poisson integral of a measure $\mu$ on $\T^2$ can then be
expressed in the form
$$
P\mu (r\tau)=\int_{\T^2}P_r(\tau\overline{w})\mu (dw),
$$
where
$\tau\overline{w}:=(\tau_1\overline{w_1},\tau_2\overline{w_2}).$
For every finite Borel measure $\mu$ and every $\tau\in \T^2$,
the Poisson maximal operator is defined by setting  $P_*|\mu
|(\tau ):=\sup_{r\in (0,1)} P_r|\mu |(\tau )$. The following
estimate is immediate.
\begin{lemma}\label{le1}
We have $P_r(\tau)\leq 16 (d(\tau,(1,1)))^{-2}$ for $r\in (0,1).$
In particular, if $s=d(\tau,\supp(\mu))>0,$ then $P_*\mu (\tau
)\leq 16s^{-2}\|\mu \| .$
\end{lemma}

Let $g:\T^2\to\real$ be a strictly positive, integrable, and lower
semicontinuous function. Following Rudin, we may express it as
\[g=\sum_{j=1}^\infty p_j,\] where the $p_j$ are non-negative
trigonometric polynomials on $\T^2.$ For each $j\geq 1$, Rudin
shows that one may  choose a positive singular measure $\mu_j$
with $\mu_j (\T^2)=\int_{\T^2}p_j\, dm_2$ and so that
$P(p_j-\mu_j)$ is the real part of an analytic function on $\D^2.$
More specifically,  $\mu_j$ is chosen to be of the form
$p_j\lambda_{k_j}$, where $k_j\geq {\rm deg}(p_j)$ and for any
positive integer $k$ the measure $\lambda_k$ has the Fourier
series expansion \beqla{eq1}
\lambda_k=\sum_{j=-\infty}^{\infty}\exp (ikj(\theta_1+\theta_2))
\eeq on $\T^2$, where  $(\theta_1,\theta_2)$ corresponds to the
point $(e^{i\theta_1},e^{i\theta_2})$ on $\T^2.$ This measure is
positive, has mass one, and with respect to the standard Euclidean
identification $\T^2=[0,2\pi)^2$ of the 2-torus, it is just the
normalized 1-measure supported on $2k-1$ line segments of
$\T^2=[0,2\pi)^2$ parallel to the direction $(1,-1)$. On the
torus, its support consists of $k$ equally spaced closed
``rings''.

For  $s\in \complex^+:=\{ z\in\complex :\Re z\geq 0\}$, we set
 $\phi (s):=(2^{-s}, 3^{-s})$. The induced boundary map
takes the form $\phi(it)=\exp (-i\log (2)t, -i\log (3)t).$ We
denote the image of the  boundary by $L$. Thought of as a subset
of $[0,2\pi)^2$, $L$ consists of a dense set of segments that have
common direction vector $v_0:=(\log (2), \log (3)).$

\begin{lemma}\label{le2}
Let a summable sequence of nonnegative numbers $a_k$
{\rm (}$k=1,2,\ldots ${\rm )} be given. If the measure $\mu$ satisfies
$$
0\leq\mu\leq \sum_{k=1}^\infty a_k\lambda_k,
$$
then $\lim_{r\to 1^-}P\mu (\tau )=0$ for almost every $\tau\in L.$
\end{lemma}
\begin{proof}
It is enough to prove the claim for $\mu = \sum_{k=1}^\infty
a_k\lambda_k$. By \cite[Theorem 2.3.1]{R69}, we know that
$\lim_{r\to 1^-}P\mu (\tau )=0$ for $m_2$-a.e.  $\tau\in \T^2.$ Pick
any segment $J\subset L$ of length 1/2, say. By Fubini's theorem we
see that for almost every $s\in [0,1/2]$ the claim holds for almost
every $\tau \in J+s(1,-1).$ However, since the measure $\mu$ is
invariant with respect to the translation $\tau\to \tau+s(1,-1)$, we
see that the statement is true for every $s\in [0,1/2]$. In
particular, we have $\lim_{r\to 1^-}P\mu (\tau )=0$ for almost every
$\tau\in J$. By expressing $L$ as a countable union of such
segments, we obtain the conclusion of the lemma.
\end{proof}

Part (i) of Theorem  \ref{th1}  will be deduced from the following
lemma. \pagebreak

\begin{lemma}\label{pr1} Given $\epsi >0$, there
is an open set $U\subset\T^2$ with $m_2(U) <\epsi /2$ and a
probability measure $\mu$ on $\T^2$ such that the function
$$
h=P(\chi_U+(1/2)\chi_{U^c}) -P\mu ,
$$
is the real part of a function in the unit ball of $H^\infty(\D^2).$
Moreover, $\lim_{r\to 1^-} h(r\tau )=1$ for almost every $\tau\in L$
with respect to the Hausdorff 1-measure on $L.$
\end{lemma}

\begin{proof}
We begin by covering $L$ with a thin open strip $U$ that becomes
thinner and thinner so that $m_2 (U)<\epsi /2 .$ For example, we
may take
$$
U :=\bigcup_{t\in\real} B\big(\phi (t),\frac{\epsi}{100(1+|t|)^{2}}\big).
$$
The next step is to run Rudin's construction with respect to the
positive and lower semicontinuous function
$\chi_U+(1/2)\chi_{U^c}$. Thus we choose strictly positive
trigonometric polynomials $p_1,p_2,\ldots$ on $T_2$ in such a way
that $\sum_{j=1}^\infty p_j=\chi_U+(\epsi/2)\chi_{U^c}$ at every
point of $\T^2.$ Moreover, by a compactness argument, we observe
that we may perform the selection in such a way that
 \beqla{eq5} 0\ < p_j(\tau )\leq j^{-2} \quad \mbox{if}
\quad d(\tau,\partial U)\geq j^{-1}. \eeq We may also require that $
\int_{\T^2} p_j\, dm_2\leq j^{-2}. $ We set
$\mu_j=p_j\lambda_{k(j)}$ and observe that \beqla{eq48} \|
\mu_j\|=\int_{\T^2} p_j\, dm_2\leq j^{-2}. \eeq Write \beqla{eq45}
\lambda_0:=\sum_{j=1}^\infty j^{-2}\lambda_{k(j)}.\nonumber \eeq
Then, according to Lemma \ref{le2}, we have
 \beqla{eq46} \lim_{r\to
1^-}P\lambda_0(r\tau )= 0\quad \mbox{for}\;\tau\in L\setminus E,
\eeq
 where $E$ has linear measure zero. A fortiori, we have in
particular that \beqla{eq4} \lim_{r\to 1^-}P\mu_j(r\tau )=0\quad
\mbox{for}\;\tau\in L\setminus E. \eeq

We now set $\mu=\sum_{j=1}^\infty \mu_j.$ The fact that $h:=
P(\chi_U+(\epsi/2)\chi_{U^c}) -P\mu $ is the real part of an
analytic function in the unit ball of $H^\infty(\D^2)$ is immediate
from Rudin's theorem \cite[Theorem 3.5.2]{R69}. Since $U$ is open
and the mass of the two-dimensional Poisson kernel concentrates on
any neighborhood of the origin as $r\to 1^-$, we see that
$\lim_{r\to 1^-}P(\chi_U+(\epsi/2)\chi_{U^c})(rw)=1$ for every $w\in
U.$ Hence it remains to verify that $ \lim_{r\to 1^-}P\mu(r\tau )\to
0$ for almost every $\tau\in L$ with respect to Hausdorff 1--measure
on $L$. In fact, we will show that \beqla{eq6} \lim_{r\to
1^-}P\mu(r\tau )= 0\quad \mbox{if }\;\tau\in L\setminus E, \eeq
which clearly suffices.

Fix an arbitrary $\tau\in L\setminus  E. $ Write
 $s=d(\tau,\partial U)>0$, $B= B(\tau ,s/2)$, and set
\beqla{eq7} \mu^a_k:= \chi_{B}\mu_k \quad\mbox{and}\quad \mu^b_k:=
\mu_k-\mu^a_k.\nonumber \eeq Pick $k_0\geq (s/2)^{-1} $. We clearly
have \beqla{eq8} \sum_{k=k_0}^\infty \mu^a_k\leq \lambda_0 \eeq so
that (\ref{eq46}) implies that \beqla{eq9} \lim_{r\to
1^-}P(\sum_{k=k_0}^\infty  \mu^a_k)(r\tau )=0. \eeq On the other
hand, we have  $d(\tau,\supp (\mu^b_k))\geq s/2$ and $\|
\mu^b_k\|\leq \| \mu_k\|\leq k^{-2}.$  Hence Lemma~\ref{le2} yields
\beqla{eq10} P_*(\sum_{k=k_0}^\infty  \mu^b_k)(r\tau )\leq
64s^{-2}\sum_{k=k_0}^\infty k^{-2} \leq C(\tau)k_0^{-1}. \eeq By
\eqref{eq4}, we have \beqla{eq11} \lim_{r\to
1^-}P(\sum^{k_0-1}_{k=1} \mu_k)(r\tau )= 0. \eeq As $k_0$ can be
chosen arbitrarily large, the desired conclusion follows by
combining  this fact with (\ref{eq9}) and (\ref{eq10}).
\end{proof}

\begin{proof}[Proof of Theorem \ref{th1}]

We begin by proving part (ii) of the theorem. Let $h$ be the
function given in Lemma~\ref{pr1}, and assume that it is the real
part of the analytic function $H$ on $\D^2$. When $k$ is large
enough, the function $ R:=\exp (k(H-1)) $ satisfies $\|
R\|_{H^\infty (\D^2)} =1$ and  $\| R\|_{H^2(\D^2)}\leq \epsi$.
Moreover, its modulus has radial boundary values 1 at almost every
point of the set $L$ with respect to linear measure. It is almost
immediate from this that the function
$$
g(s):=R(\phi (s))=R(2^{-s},3^{-s})
$$
is, by construction, a singular inner function in $\C^+$ with $\|
g\|_{\Ht} <\epsi .$ The only matter that requires a little
attention, is how we conclude that $|g|$ has unimodular boundary
values almost everywhere. The point is that horizontal boundary
approach in $\C^+$ does not transfer exactly via $\phi$ to radial
approach, but instead to what we will call quasi-radial approach.
This means that $(r_1w_1,r_2w_2)\to (w_1,w_2)$, where $r_1\to 1^-$
and $r_2\to 1^-$ in such a way that the ratio $(1-r_1)/(1-r_2)$
stays uniformly bounded from above and below. However, apart from a
change of non-essential constants, our proof of Lemma~\ref{pr1}
remains valid for quasi-radial approach. This is easily verified for
Lemma \ref{le1}, and it remains true for the basic theorem
\cite[Theorem 2.3.1]{R69} on radial limits of singular measures (see
\cite[Exercise 2.3.2.(d)]{R69}). These remarks conclude the proof of
part (ii) of Theorem~1.

We now turn to the proof of part (i) of Theorem~\ref{th1}. The basic
construction is similar to the one in the proof of part (ii), so we
only indicate the required changes. To simplify the notation, we
identify the imaginary axis with $L$. Lebesgue measure on the
imaginary axis is denoted by $\nu$. This time we cover only part of
the image of the imaginary axis $L$ by an open set $U.$ To this end,
given $\epsi
>0$, we first construct by induction a sequence of open subsets
$U_1,U_2,\ldots \subset T^2$ with the following properties for each
$n\geq 1$:
\begin{itemize}
\item[(1)] There is $t_n\geq n$ so that $\nu (U_n\cap [0,it_n])> (1-\epsi /2 )t_n.$
\item[(2)]  The closures $\overline{U_1},\overline{U_2},\ldots ,\overline{U_n}$ are disjoint.
\item[(3)]  The set $U_n$ is a finite union of open dyadic squares and
$$
\sum_{j=1}^ nm_2(U_j)<\epsi /2 .
$$
\end{itemize}
In the first step, we set $t_1=1$ and, apart from a finite number of
points, we cover $[0, it_1]$ by a finite union of dyadic open cubes
$U_1$ with $m_2 (U_1)=m_2(\overline{U_1})<\epsi /2 $. Assume then
that sets $U_1,\ldots  U_n$ with the right properties have been
found. Since we are dealing with finite unions of open squares, it
holds that $m_2(\overline{\bigcup_{j=1}^ nU_j})\leq \sum_{j=1}^
nm_2(U_j)<\epsi /2$ and hence we may apply the continuous version of
Weyl's equi-distribution theorem for Kronecker flows in order to
select  $t_{n+1}\geq n+1$ with
$$
\nu \big((\overline{\cup_{j=1}^ nU_j})\cap [0,it_{n+1}]\big)<\epsi
/2.
$$
Then $U_{n+1}$ is obtained by covering a sufficiently large
portion of the set $[0,it_{n+1}]\setminus \overline{\bigcup_{j=1}^
nU_j}$ by a union of open dyadic squares that has a positive
distance to $\overline{\bigcup_{j=1}^ nU_j}$ and satisfies
$m_2(U_{n+1})<\epsi - \sum_{j=1}^ nm_2(U_j).$ This completes the
induction.

Set $U=\bigcup_{k=1}^ \infty U_{2k-1}$ and $V=\bigcup_{k=1}^ \infty
U_{2k}$. We run the Rudin construction exactly as in the proof of
part (ii) corresponding to the lower semicontinuous boundary
function $\chi_U+(\epsi /2)\chi_{U^c}.$ Hence, we obtain a
polyharmonic function $h$ on $\D^2$ with (quasi-)radial boundary
values $1$ at almost every point of $U\cap L$ (respectively $\epsi
/2$ at almost every point of $V$), and such that $h$ is the real
part of the analytic function $H$ on $\D^2$. By property (1) of the
sets $U_1,U_2,\ldots $, it is then evident that with sufficiently large $k$ the function
$f(s):=\exp (k(H(2^{-s},3^{-s})-1))$ satisfies
$f\in\Hi$ and
\[ \liminf_{T\to \infty} \frac{1}{T} \int_0^T |f(it)|^2 dt \le \epsi   \]
as well as
\[ \limsup_{T\to \infty} \frac{1}{T} \int_0^T |f(it)|^2 dt =1. \]
\end{proof}

\begin{question} \label{boundaryf}
Which are the functions $\psi$ such that $\psi=|f(it)|$ almost
everywhere for some function $f$ in $\Hi$?
\end{question}
This is no doubt a difficult question, because we do not even have a
description of the radial limits of $|F|$ for $F$ in
$H^\infty(\D^2)$. A loose restatement of the question is as follows:
How much of the almost periodicity of $|f(\sigma+it)|$ on the
vertical lines in $\C_+$ is carried to the boundary limit function
$|f(it)|$?

\section{The embedding problem for $\Hp$}\label{se:embedding}

We begin by recalling the definition of the spaces $\Hp$. We will
use standard multi-index notation, which means that if
$\beta=(\beta_1,\ldots, \beta_k,0,0,\ldots)$, then $z^\beta
=z_1^{\beta_1}\cdot\ldots \cdot z_k^{\beta_k}$. If $p>0$ and
$F(z)=\sum b_\beta z^\beta$ is a finite polynomial in the variables
$z_1,z_2,\ldots$, its $H^p$ norm is  \beqla{eq:2.1} \|
F\|_{H^p(\D^\infty )}:=\left(\int_{\T^\infty}|F(\tau)|^p\, dm_\infty
(\tau )\right)^{1/p},\nonumber \eeq where $m_\infty$ is the Haar
measure on the distinguished boundary $\T^\infty .$ The space $H^p
(\D^\infty )$ is obtained by taking the closure of the set of
polynomials with respect to this norm (quasi-norm in case $p\in (0,1)$). For $p\geq 1$, it consists of
all analytic elements in $L^p(\T^\infty )$, i.e., all functions in
$L^p$ for which all Fourier coefficients with at least one negative
index vanish. Obviously, $\| F\|^2_{H^2(\D^\infty
)}=\sum_{\beta}|b_\beta|^2.$

Now let $f(s)=\sum_{n=1}^m a_nn^{-s}$ be a finite Dirichlet
polynomial. By the Bohr correspondence, $f$ lifts to the
polynomial $F(z)=\sum b_\beta z^\beta$ on $\D^\infty ,$ where
$b_\beta =a_n$, given that $n$ has the prime factorization
$n=p_1^{\beta_1}p_2^{\beta_2}\cdot \ldots \cdot p_k^{\beta_k}$;
here $p_1=2, p_2=3, \ldots$ are the primes listed in increasing
order. We define $\| f\|_{\Hp}:=\| F\|_{H^p (\D^\infty )}.$ The
space $\Hp$ is obtained \cite{B00} by taking the closure of the
Dirichlet polynomials with respect to this norm. Consequently, the
spaces $H^p(\D^\infty )$ and $\Hp$ are isometrically isomorphic
via the Bohr correspondence.

When $p=2$, the definition given above coincides with the original
one for the Hilbert space $\Ht$, and the Bohr correspondence $f
\leftrightarrow F$ carries over to the case $p=\infty$. In fact,
also Carlson's theorem (with $\sigma=0$) can be used to define the
$\Hp$ norm: for every Dirichlet polynomial $f$ and $p>0$ we have
\beqla{eq:2.0} \| f\|^p_{\Hp }= \lim_{T\to
\infty}\frac{1}{2T}\int_{-T}^T |f(i t)|^p\, dt. \eeq The equality
follows by an application of the ergodic theorem, since $f$ is
continuous. However, let us also sketch a more elementary proof. By
polarizing Carlson's identity \refeq{means} and applying the
resulting inner product identity to the functions $f^j$ and $f^k$
(with integers $j,k\geq 0$) we obtain
\[ \int_{\T^\infty}F^j\overline{F}^ k\, dm_\infty (\tau ) =\lim_{T\to
\infty}\frac{1}{2T}\int_{-T}^T f(i t)^j\overline{f(i t)}^k\, dt.\]
The identity \refeq{eq:2.0} is then obtained by applying
Weierstrass's theorem on polynomial approximation to the
continuous function $z\to |z|^p$ on the set  $\{ z\, :\, |z|\leq
2\|f\|_\infty\} .$

Estimates obtained by B. Cole and T. Gamelin \cite[Theorem]{CG86}
verify that point evaluations $f\mapsto f(s)$ are bounded in $\Hp$
if and only if $s$ is in the half-plane $\C^+_{1/2}=\{ s=s+it:\
\sigma>1/2\}$. The norm of the functional of point evaluation is of
order $(\sigma -1/2)^{-1/p},$ just as it is for functions in
$H^p(\C^+_{1/2})$. Hence elements of $\Hp$ are analytic in
$\C^+_{1/2}$ with uniformly converging Dirichlet series in any
half-plane $\sigma  \geq 1/2+\varepsilon$ We refer to \cite{B00} for
additional information about the spaces $\Hp$.

A question of primary importance concerning $\Hp$, first considered
by Bayart \cite{B00}, is whether the analogue of the embedding
result \refeq{embedding} holds for $p\neq 2$. It suffices to
formulate the question only for polynomials, since existence of
non-tangential boundary values almost everywhere would be an
immediate consequence of a positive answer, and the inequality could
then be stated for all elements in $\Hp.$

\begin{question}[The embedding problem]\label{question}
Fix an exponent $p>0$, that is not an even integer. Does there
exist a constant $C_p <\infty$ such that
\begin{equation} \label{p-embedding} \int_{0}^{1}\left|f\left(\frac{1}{2}+i
t\right)\right|^p\,  dt \le C_p \|f\|_{\Hp}^p
\end{equation}
for every Dirichlet polynomial $f$?
\end{question}

We have excluded the case $p=2k$ with $k\in\nanu$ because then the
answer is trivially positive: Just apply the case $p=2$ to the
function $f^k$ in $\Ht$. This observation\footnote{In \cite{B00},
Bayart proclaimed a positive answer to Question 6 for $p>2$.
Unfortunately, his proof, based on this observation and an
interpolation argument, contains a mistake.} provides evidence in
favor of a positive answer. The growth estimates for functions in
$\Hp$ mentioned above point in the same direction. An answer to
Question \ref{p-embedding} seems to be a prerequisite for a further
development of the theory of the spaces $\Hp$.

Let us now point at some properties of the spaces $\Hp$---no doubt
known to specialists---indicating that Question \ref{question} is
deep and most probably very difficult. First of all, it is easily
seen that for $p>1$ the isometric subspace $\Hp (\D^\infty
)\subset L^p(\T^\infty )$ is not complemented in $L^p(\T^\infty )$
unless $p=2.$ Namely, if there were a bounded projection, one
could easily apply the Rudin averaging technique to show that the
$L^2$-orthogonal projection is bounded in $L^p$. In other words,
the infinite product of one-dimensional Riesz projections would be
bounded in $L^p$. A fortiori, the only possibility is that the
norm of the dimensional projection is one (simply consider
products of functions each depending on one variable only), i.e.
$p=2.$ This fact makes it difficult to apply interpolation between
the already known values $p=2,4,6,\ldots$. Moreover, similar
arguments show that, in the natural duality, we have
${\mathscr{H}}^{p'}\subset (\Hp)'$, but the inclusion is strict
whenever $p\not=2.$ In fact, for $p\in (1,2)$ one has
\begin{equation}\label{duality}
\mbox{if} \quad p\in (1,2),\;\;\mbox{then}\quad (\Hp)'\subset
{\mathscr{H}}^{q}\quad \mbox{if and only if} \;\; q\leq 2.
\end{equation}

There are some famous unresolved conjectures in analytic number
theory, due to H. Montgomery, that deal with  norm inequalities
for Dirichlet polynomials (see \cite[pp. 129, 146]{M94} or
\cite[p. 232--235]{IK05}). One of Montgomery's
conjectures states that for every $\varepsilon
>0$ and $p\in (2,4)$
there exists $C=C(\varepsilon )$ such that for all finite
Dirichlet polynomials $f=\sum_{n=1}^ Na_nn^{-s}$ with $|a_n|\leq
1$ one has
 \begin{equation} \label{montgomery} \int_{0}^{T}\left|f\left(i
t\right)\right|^p\,  dt \le CN^{p/2+\varepsilon}(T+ N^{p/2})\quad \mbox{for}\;\; T>1.
\end{equation}
If true, this inequality would imply the density hypothesis for
the zeros of the Riemann zeta function. It is quite interesting to
note that \refeq{montgomery} is also known to be true for $p=2,4$
(or any even integer). The similarities suggest for a possible
connection between Montgomery's conjectures and our embedding
problem. Although it appears to be difficult to give a precise
link, both problems can be understood as dealing with the ``degree
of flatness'' of Dirichlet polynomials.

As a first step towards a solution of the embedding problem, one
could ask for a weaker partial result:
\begin{question}\label{question2}
Assume that $2<q<p<4.$ Is it true that \[
\left(\int_{0}^{1}\left|f\left(\frac{1}{2}+i t\right)\right|^q \,
dt\right)^{1/q} \le C_q \|f\|_{{\mathscr{H}}^p}?\] Is this true at
least for one such pair of exponents?
\end{question}
Let us denote by $A$ the adjoint operator of the natural embedding
operator. Thus  for functions $g$ on $[0,1]$ one has
$$
Ag =\sum_{n=1}^\infty \big(n^{-1/2} \widehat g(\log n)\big)
n^{-s},
$$
where $\widehat{g}$ is the Fourier transform of $g$. Observe that,
due to \refeq{duality}, the existence of the embedding
\refeq{p-embedding} does not imply that $A:L^{p'}(0,1)\to
{\mathscr{H}}^{p'}$. However, a positive answer to the following
question would, by interpolation, imply that for each $p>2$ there is
$q>2$ such that the embedding operator acts boundedly from $\Hp$ to
$L^q(0,1)$.
\begin{question}\label{question3}
Is there an exponent $r\in (1,2)$ such that $A:L^{r}\to
H^{1}(\D^\infty )$ is bounded?
\end{question}

\section{Fatou theorems for $\Hp$}
We will now in some sense return to what appeared as a difficulty in
the proof of Theorem~1, namely that the imaginary axis has measure
zero when viewed as a subset of $\T^2$. Thus, a priori, it makes no
sense to speak about the restriction to the imaginary axis of a
function in $L^p(\T^\infty)$. We will now show that, for functions
in $H^p(\D^\infty )$, we can find a meaningful connection to the
boundary limits of the corresponding Dirichlet series.

We consider a special type of boundary approach by setting for each
$\tau =(\tau_1,\tau_2,\ldots )\in\T^\infty$ and $\theta \geq 0$
$$
b_\theta (\tau) :=(p_1^{-\theta}\tau_1,p_2^{-\theta}\tau_1,\ldots ).
$$
We also recall that the Kronecker flow on $\overline{\D}^\infty$ is defined by setting
$$
T_t((z_1,z_2,\ldots )):=(p_1^{-it}z_1,p_2^{-it}z_2,\ldots ).
$$
 For an arbitrary $z\in\overline{\D}^\infty$,
we denote by $T(z)$ the image of $z$ under this flow, i.e., $T(z)$
is the one-dimensional complex variety $T(z):= \{T_t(z): t\in\real
\}$. We equip $T(z)$ with the natural linear measure, which is just
Lebesgue measure on the real $t$-line. Moreover, for $\sigma
>0$, we set $\T^\infty_\sigma :=b_\sigma (\T^\infty )$, which is the
distinguished boundary of the open set $ b_\sigma ( \D^\infty ).$
The natural Haar measure $m_{\infty,\sigma}$ on $\T^\infty_\sigma$
is obtained as the pushforward of $m_\infty$ under the map $b_\theta
$. The set $\T^\infty_{1/2}$ is of special interest, since in a
sense it serves as a natural boundary for the set
$\D^\infty\cap\ell^2$, where point evaluations are bounded for the
space $H^p(\D^\infty )$ with $p\in (0,\infty )$.

Our version of Fatou's theorem for $H^\infty$ reads as follows.

\begin{theorem}\label{fatouinf}
Let $F$ be a function in $H^\infty (\D^\infty )$. Then we may pick a
representative $\widetilde{F}$ for the boundary function of $F$ on
the distinguished boundary $\T^\infty$ such that
$\widetilde{F}(\tau)=\lim_{\theta\to 0^+}F(b_\theta (\tau ))$ for
almost every $\tau\in\T^\infty .$ In fact, for every
$\tau\in\T^\infty $, we have
 $\widetilde{F}(\tau')=\lim_{\theta\to 0^+}F(b_\theta (\tau' ))$
 for almost every $\tau'\in T(\tau )$.
\end{theorem}

\begin{proof}
Recall that by \cite{CG86} the values of  $H^\infty (\D^\infty
)$-functions are well-defined in $\D^\infty$ at points $z$ with
coordinates tending to zero, i.e. for $z\in c_0.$ We simply define
the desired representative $\widetilde F$  for the the boundary
values by setting $\widetilde F(\tau)=\lim_{\theta\to
0^+}F(\theta\circ \tau )$ whenever this limit exists and otherwise
$\widetilde F(\tau)=0$. The Borel measurability of $\widetilde F$ is
clear. The second statement follows immediately by considering for
each $\tau\in\T^\infty $ the analytic function $f_\tau:
f_\tau(\theta +it)= F(T_tb_\theta (\tau ) )$ and observing that for
each $\tau\in\T^\infty $ we have $f_\tau\in\Hi$. Now the classical
Fatou theorem applies to $f_\tau$. The fact that the set
$\{\tau\in\T^\infty : \; \lim_{\theta\to 0^+}F(b_\theta ( \tau )
)\,\mbox{exists}\}$ has full measure is an immediate consequence of
the ergodicity of the Kronecker flow $\{T_t\}_{t\geq 0}$ and the
second statement. Finally, we  observe that it is easy to check  the
formula
\[\widehat F(\beta )=p_1^{\beta_1\sigma}\cdot\ldots\cdot
p_k^{\beta_k\sigma}\int_{\T^\infty} F(b_\theta (\tau
))\overline{\tau}^\beta m_\infty (d\tau )\] for the Fourier
coefficients of an $H^\infty (\D^\infty )$-function. Lebesgue's
dominated convergence theorem now yields $\widehat{\widetilde
F}=\widehat F$, whence $\widetilde F =F$ almost surely, and this
finishes the proof of the first statement.
\end{proof}

To arrive at a similar result for $\Hp$, we need to make sense of
the restriction $F\mapsto F|_{\T^\infty_{1/2}}$ as a map from $H^p
(\D^\infty )$ to $L^p(\T^\infty_{1/2},m_{\infty, 1/2})$. When $F$ is
a polynomial, we must have
\[ F|_{\T^\infty_{1/2}}(\tau)=F(b_{1/2}(\tau)). \]
Since this formula can be written as a Poisson integral and the
polynomials are dense in $H^p (\D^\infty )$, this leads to a
definition of $F|_{\T^\infty_{1/2}}$ for general $F$. Indeed, by
using elementary properties of Poisson kernels for finite polydiscs,
we get that $F\mapsto F|_{\T^\infty_{1/2}}$ is a contraction from
$H^p (\D^\infty )$ to $L^p(\T^\infty_{1/2},m_{\infty, 1/2})$.

\begin{theorem}\label{fatoup}
Let $F$ be a function $H^p (\D^\infty )$ for $p\ge 2$. Then we may
pick a representative $\widetilde{F}_{1/2}$ for the restriction
$F|_{\T^\infty_{1/2}}$ on the distinguished boundary $\T^\infty$
such that $\widetilde{F}(\tau)=\lim_{\theta\to 1/2^+}F(b_\theta
(\tau ))$ for almost every $\tau\in\T^\infty .$ In fact, for every
$\tau\in\T^\infty $, we have
 $\widetilde{F}_{1/2}(\tau')=\lim_{\theta\to 1/2^+}F(b_\theta (\tau' ))$
 for almost every $\tau'\in T(\tau )$.
\end{theorem}

\begin{proof} The existence of the boundary values is
obtained just as in the proof of Theorem~\ref{fatouinf}. This time
one applies the known embedding for $p=2$ to define
$\widetilde{F}_{1/2}$.
\end{proof}

We may now observe that if $F$ is in $H^p(\D^\infty)$ ($p\ge 2$) and
the embedding \eqref{p-embedding} holds, then we have for every
$\tau\in \T^\infty_{1/2}$
\begin{equation}\label{p-Carlson}
\lim_{T\to\infty} \frac{1}{T}\int_{0}^T |\widetilde{F}(T_t\tau
)|^p\, dt=\|\widetilde{F}_{1/2}\|^p_{L^p(\T^\infty )}.
\end{equation}
Indeed, \eqref{p-Carlson} holds for polynomials. Hence, employing
\eqref{p-embedding} and the fact that polynomials are dense in
$H^p(\D^\infty$, we obtain \eqref{p-Carlson}.

On the other hand, if \refeq{p-Carlson} is true, then by the closed
graph theorem (fix $T=1$), the embedding \refeq{p-embedding}
follows. We have therefore made sense of the statement that the
``p-Carlson identity'' \refeq{p-Carlson} is equivalent to the
embedding \refeq{p-embedding}.

It is rather puzzling that \eqref{p-Carlson}, which may be
understood as a strengthened variant of the Birkhoff--Khinchin
ergodic theorem for functions in $H^p(\D^\infty)$, is known to hold
only when $p=2,4,6,...$.

\medskip

\noindent{\sl Acknowledgement.} We thank P. Lindqvist for useful
discussions concerning Section \ref{se:embedding}.

\end{document}